\documentclass[12pt,a4paper]{amsart}
\usepackage{amsmath,amssymb,amsthm,mathtools}
\usepackage{amsfonts,amsxtra,amsthm,tikz}
\usepackage{caption}
\usepackage{capt-of}
\usepackage[headings]{fullpage}
\usepackage[plain]{algorithm}
\usepackage{soul}
\setstcolor{red}

\makeatletter
\renewcommand{\ALG@name}{\footnotesize Algorithm}

\RequirePackage[colorlinks=false]{hyperref}
\usepackage[]{algpseudocode}

\usetikzlibrary{arrows,positioning}
\usetikzlibrary{decorations.markings}
\tikzstyle arrowstyle=[scale=1.5]
\tikzstyle directed=[postaction={decorate,decoration={markings,
    mark=at position .625 with {\arrow[arrowstyle]{stealth}}}}]

\newtheorem{theorem}{Theorem}[section]
\newtheorem{proposition}[theorem]{Proposition} 
\newtheorem{lemma}[theorem]{Lemma}              

\theoremstyle{definition}
\newtheorem{definition}[theorem]{Definition}        

\theoremstyle{remark}
\newtheorem{example}[theorem]{Example}          

\numberwithin{equation}{section}

\newcommand\cx[1]{\multicolumn{1}{|c|}{\!\!#1\!\!}}
\newcommand{\bnv}{\textup{\texttt{burnt\_vertices}}}
\newcommand{\pp}{\textup{\texttt{parking\_place}}}
\newcommand{\spp}{\textup{\texttt{street\_parking}}}
\newcommand{\ff}{\textup{\texttt{first\_free}}}
\newcommand{\damp}{\texttt{dampened\_edges}}
\DeclareMathOperator{\cN}{\textup{\texttt{neighbours}}}
\newcommand{\tree}{\texttt{tree\_edges}}
\newcommand\op{\textup{\texttt{occupied\_positions}}}

\newcommand\bak{\mathbf{\tilde{a}}^{k}}
\newcommand\bat[1]{\mathbf{\widetilde{a}}^{#1}}
\newcommand\cG{\mathcal{G}}

\newcommand\cR{\mathcal{R}}
\newcommand\bG{\overline{\mathcal{G}}}
\newcommand\oA{\overline{A}}
\newcommand\oV{\overline{V}}
\newcommand\bGk{\overline{\mathcal{G}_k}}

\newcommand\cA{\mathcal{A}}

\DeclareMathOperator{\Ish}{Ish}
\DeclareMathOperator{\Shi}{Shi}

\newcommand\gammab[1]{\gamma\mkern-2mu\raisebox{-1.5pt}{${}_{#1}$}}
\newcommand\gammad[1]{\scriptstyle\gamma\mkern-2mu\raisebox{-1.25pt}{$\scriptscriptstyle{}_{#1}$}}
\newcommand\ba{\mathbf{a}}
\newcommand\ta{\tilde{a}}
\newcommand\bb{\mathbf{b}}
\newcommand\be{\mathbf{e}}
\newcommand\bw{\mathbf{w}}
\newcommand\bx{\mathbf{x}}
\newcommand\bt{\mathbf{t}}
\newcommand\Reals{\mathbb{R}}
\newcommand\N{\mathbb{N}}
\newcommand\Z{\mathbb{Z}}

\newcommand{\PF}{\mathsf{PF}}

\DeclareMathOperator{\dfs}{DFS}

\title{Partial Parking Functions}

\author[R.~Duarte]{Rui Duarte} \address{CIDMA and Department of Mathematics, University of Aveiro, 3810-193 Aveiro, Portugal}
\email{rduarte@ua.pt}

\author[A.~Guedes~de~Oliveira]{Ant\'onio Guedes de Oliveira}
\address{CMUP and Department of Mathematics, Faculty of Sciences, University of Porto, 4169-007 Porto, Portugal}
\email{agoliv@fc.up.pt}

\begin{document}

\begin{abstract}
We characterise the Pak-Stanley labels of the regions of a family of  hyperplane arrangements 
that interpolate between the Shi arrangement and the Ish arrangement. 
\end{abstract}
\maketitle

\section{Introduction}
\noindent
In this paper, we characterise the Pak-Stanley labels of the regions of the recently introduced family of the arrangements of hyperplanes ``between Shi and Ish'' (cf. \cite{DGO3}). 

In other words, for $n\in\N=\{1,2,\dotsc\}$
there is a labelling (due to Pak and Stanley \cite{Stan2}) of the regions of the $n$-dimensional Shi arrangement (that is, the connected components of the complement in $\Reals^n$ of the union of the hyperplanes of the arrangement) by the $n$-dimensional \emph{parking functions}, and the labelling in this case is a bijection. Remember that the parking functions can be characterised (see Definition~\ref{def31} below;
as usual, given $n\in\N\cup\{0\}$, we define $[n]:=[1,n]$ where $[m,n]:=\{i\in\Z\mid m\leq i\leq n\}$)  as
\begin{verse}
$\ba=(a_1,\dotsc,a_n)\in[n]^n$ such that there is a permutation $\sigma\in\mathfrak{S}_n$ with 
\begin{equation*}
a_{\sigma(i)}\leq i, \text{ for every } i\in[n]\,.
\end{equation*}
\end{verse}
By labelling under the same rules the regions of the $n$-dimensional \emph{Ish arrangement},  we obtain a new bijection between
these regions and the so-called \emph{Ish-parking functions} \cite{DGO2} which can be characterised (see Theorem~\ref{def.ipf} below) as
\begin{verse}
$\ba=(a_1,\dotsc,a_n)\in[n]^n$ such that there is a permutation $\sigma\in\mathfrak{S}_n$ with 
\begin{equation*}
\begin{cases}a_{\sigma(i)}\leq i, \text{ for every } i\in[a_1]\,;\\
\sigma(i+1)< \sigma(i), \text{ for every } i\in[a_1-1]\,.\end{cases}
\end{equation*}
\end{verse}
In this paper, we show that the sets of labels corresponding to the arrangements $\cA^k_n$ ($2\leq k\leq n$) that interpolate between the Shi and the Ish arrangements (which are $\cA^2_n$ and $\cA_n^n$, respectively) can be characterised (see Proposition~\ref{def2.ppf}) as
\begin{verse}
$\ba=(a_1,\dotsc,a_n)\in[n]^n$ such that there is a permutation $\sigma\in\mathfrak{S}_n$ with 
\begin{equation*}
\begin{cases}
a_{\sigma(i)}\leq i \text{ for every } i\in[a_1]\text{ and for every }
i\in[k,n]\text{ such that }\sigma(i)\geq k\,;\\
\sigma(i+1)< \sigma(i) \text{ for every } i\in[a_1-1]\text{ such that }\sigma(i)<k\,.\end{cases}
\end{equation*}
\end{verse}
We call these sets of labels \emph{partial parking functions} and note that they all have the same number of elements, viz. $(n+1)^{n-1}$, by \cite[Section~2 and Theorem~3.7]{DGO}.
Note that if $k=2$, $\ba$ satifies the first condition above and $i<a_1$ verifies $\sigma(i)<k$, then $a_1=a_{\sigma(i)}\leq i$, a contradiction.
\section{Preliminaries}
\noindent
Consider, for a natural number $n\geq3$, hyperplanes of $\Reals^n$ of the following three types.
Let, for $1\leq i<j\leq n$,
\begin{align*}
&C_{ij}=\big\{(x_1,\dotsc,x_n)\in\Reals^n\mid x_i=x_j\big\}\,,\\
&S_{ij}=\big\{(x_1,\dotsc,x_n)\in\Reals^n\mid x_i=x_j+1\big\}\,,\\
&I_{ij}=\big\{(x_1,\dotsc,x_n)\in\Reals^n\mid x_1=x_j+i\big\}\\
\intertext{and define, for $2\leq k<n$,}
&\cA^k_n:=\big\{C_{ij} \mid 1\leq i<j\leq n \big\}\\
&\hphantom{\cA^k_n:={}}\cup\big\{I_{ij} \mid 1\leq i<j\leq n \,\wedge\, i<k \big\}\\
&\hphantom{\cA^k_n:={}}\cup\big\{S_{ij} \mid k\leq i<j\leq n\big\}
\end{align*} 	
Note that $\cA^2_n=\Shi_n$, the $n$-dimensional Shi arrangement, and $\cA_n^n=\Ish_n$,
the $n$-dimensional Ish arrangement   introduced by Armstrong \cite{Arm}.

\subsection{The Pak-Stanley labelling}~\\
Let $\cA=\cA^k_n$ and define,
for every $(i,j)$ with $1\leq i<j\leq n$,
$$m_{ij}=\left\{\begin{tabular}{r}
$0$, \hfill if no hyperplane of equation $x_i-x_j=a$ belongs to $\cA$;\\
$\max\{a\mid \text{$\cA$ contains a hyperplane of equation $x_i-x_j=a$}\}$,
\hfill\text{ otherwise.}\end{tabular}\right.$$
Note that
\begin{itemize}
\item there are  no hyperplanes of equation $x_i -x_j =a$ with $a>0$ and $i>j$; 
\item if $a>0$ and the hyperplane of equation $x_i -x_j =a$ belongs to $\mathcal{A}$,
then it also belongs to $\mathcal{A}$ the hyperplane of equation $x_i -x_j =a-1$.
\end{itemize}
Similarly to what Pak and Stanley did for the regions of the Shi arrangement (cf. \cite{Stan2}),
we may represent a region $\cR$ of $\cA$ as follows.

Suppose that $\bx=(x_1,\dotsc,x_n)\in\cR$ and $x_{w_1}>\dotsb>x_{w_n}$ for a given 
$\bw=(w_1,\dotsc,w_n)\in\mathfrak{S}_n$.
Let $\mathcal{H}$ be the set of triples $(i,j,a_{ij})$ such that $i,j,a_{ij} \in \mathbb{N}$, $1\leq i < j \leq n$, $x_i > x_j$, $a_{ij}-1 < x_i - x_j < a_{ij}$ 
and the hyperplane of equation
$ x_i - x_j = a_{ij}$ belongs to $\mathcal{A}$, and let 
$$\mathcal{I}=\big\{(i,j)\in\N^2\mid 1\leq i<j\leq n \text{  and  } (i,j,a)\notin\mathcal{H}\text{ for every $a\in\N$}\big\}\,.$$
Then, 
\begin{equation}\label{eqr}
\cR=\left\{(x_1,\dotsc,x_n)\in\Reals^n\middle|
\begin{array}{l}x_{w_1}>x_{w_2}>\dotsb>x_{w_n},\\ 
a_{ij}-1<x_i-x_j<a_{ij}\,,\  \forall (i,j,a_{ij})\in \mathcal{H}\\
x_i-x_j>m_{ij}\,,\  \forall (i,j)\in \mathcal{I}
\end{array} \right\}\,.\end{equation}
We represent $\cR$ by $\bw$, decorated with one \emph{labelled arc} for each triple of $\mathcal{H}$, as follows. Given $(i,j,a_{ij})\in\mathcal{H}$, 
the arc  connects $i$ with $j$ and is labelled $a_{ij}$, with the following exceptions:
if $i\leq j<p\leq m$,  $(i,m,a_{im}),(j,p,a_{jp})\in\mathcal{H}$ and $a_{jp}=a_{im}$,
then we omit the arc connecting $j$ with $p$.
Note that, given $i\leq j<p\leq m$, forcibly
$$a_{im}\ >x_i-x_m\geq x_i-x_p\ \geq x_j-x_p$$
and so $a_{im}\geq a_{jp}$.
In the left-hand side of Figure~\ref{isharr} the regions of $\Ish_3$ are thus represented.

The Pak-Stanley labelling of these regions may be defined as follows.
As usual, let  $\be_i$ be the $i$.th element of the standard basis of $\Reals^n$, $\be_i=(0,\dotsc,0,1,0,\dotsc,0)$.

\begin{definition}[Pak-Stanley labelling \cite{Stan2}, ad.]
Let $\cR_0$ be the region defined by
$$x_n+1>x_1>x_2>\dotsb>x_n$$
(bounded by the hyperplanes of equation $x_j=x_{j+1}$ for $1\leq j<n$ and by the hyperplane of equation $x_1=x_n+1$).
Then label $\cR_0$ with $\ell(\cA^k_n,\cR_0):=(1,\dotsc,1)$, and,
 given two regions $\cR_1$ and $\cR_2$ separated by a unique hyperplane $H$ of $\cA^k_n$ such that $\cR_0$ and
$\cR_1$ are on the same side of $H$, label the regions $\cR_1$ and $\cR_2$ so that
$$\ell(\cA^k_n,\cR_2)=\ell(\cA^k_n,\cR_1)+\begin{cases}\be_i,&\text{if $H=C_{ij}$ for some $1\leq i<j\leq n$;}\\\be_j,&\text{if $H=S_{ij}$ or $H=I_{ij}$ for some $1\leq i<j\leq n$.}\end{cases}$$
\label{def.PSl}\end{definition}

\begin{figure}[t]
\noindent
\strut\hfill
\begin{tikzpicture}[scale=1.1]
\draw  [very thick]  (-0.866, -1.5) -- (1.732, 3.)  node [pos=1.,above,sloped]  {\scriptsize$x=z$};
\draw  [very thick]  (-1.732, 3.) -- (0.866, -1.5)  node [pos=1.05,above,sloped]  {\scriptsize$y=z$};
\draw  [very thick]  (-3., 0.) -- (2., 0.) node [pos=1.05,above,sloped]  {\scriptsize$x=y$};
\draw  [thick,dashed]  (-2.866, -1.5) -- (-0.268, 3.)  node [pos=1.1,above,sloped]  {\scriptsize$x=z+2$};
\draw  [thick,dashed]  (-1.866, -1.5) -- (0.732,  3.)  node [pos=1.1,above,sloped]  {\scriptsize$x=z+1$};
\draw  [thick,dashed]  (-3., 0.866) -- (2., 0.866) node [pos=1.1,above,sloped]  {\scriptsize$x=y+1$};
\node at (-1.,2.533){\begin{tikzpicture}[scale=0.2,baseline=-3.5pt]%
\draw (0,0) node[inner sep=.5mm,minimum size=2mm](x1){\small$1$};
\draw (1,0) node[inner sep=.5mm,minimum size=2mm] (x3){\small$3$};
\draw (2,0) node[inner sep=.5mm,minimum size=2mm] (x2){\small$2$};
\end{tikzpicture}};
\node at (0.,2.533){\begin{tikzpicture}[scale=0.2,baseline=-3.5pt]%
\draw (0,0) node[inner sep=.5mm,minimum size=2mm](x1){\small$1$};
\draw (1,0) node[inner sep=.5mm,minimum size=2mm] (x3){\small$3$};
\draw (2,0) node[inner sep=.5mm,minimum size=2mm] (x2){\small$2$};
\draw[thick,looseness=1.5]  (x1) to[out=90,in=90] (x3); 
\draw[thick,looseness=1.]  (x1) to[out=80,in=100] (x3); 
\end{tikzpicture}};
\node at (1.,2.533){\begin{tikzpicture}[scale=0.2,baseline=-3.5pt]%
\draw (0,0) node[inner sep=.5mm,minimum size=2mm](x1){\small$1$};
\draw (1,0) node[inner sep=.5mm,minimum size=2mm] (x3){\small$3$};
\draw (2,0) node[inner sep=.5mm,minimum size=2mm] (x2){\small$2$};
\draw[thick,looseness=1.]  (x1) to[out=90,in=90] (x3); 
\end{tikzpicture}};
\node at (-2.,1.433){\begin{tikzpicture}[scale=0.2,baseline=-3.5pt]%
\draw (0,0) node[inner sep=.5mm,minimum size=2mm](x1){\small$1$};
\draw (1,0) node[inner sep=.5mm,minimum size=2mm] (x2){\small$2$};
\draw (2,0) node[inner sep=.5mm,minimum size=2mm] (x3){\small$3$};
\end{tikzpicture}};
\node at (1.5,1.433){\begin{tikzpicture}[scale=0.2,baseline=-3.5pt]%
\draw (0,0) node[inner sep=.5mm,minimum size=2mm](x3){\small$3$};
\draw (1,0) node[inner sep=.5mm,minimum size=2mm] (x1){\small$1$};
\draw (2,0) node[inner sep=.5mm,minimum size=2mm] (x2){\small$2$};
\end{tikzpicture}};%
\node at (-1.,1.13){\begin{tikzpicture}[scale=0.2,baseline=-3.5pt]%
\draw (0,0) node[inner sep=.5mm,minimum size=2mm](x1){\small$1$};
\draw (1,0) node[inner sep=.5mm,minimum size=2mm] (x2){\small$2$};
\draw (2,0) node[inner sep=.5mm,minimum size=2mm] (x3){\small$3$};
\draw[thick,looseness=1.]  (x1) to[out=90,in=90] (x3); 
\draw[thick,looseness=1.25]  (x1) to[out=100,in=80] (x3); 
\end{tikzpicture}};
\node at (-2.7,0.463){\begin{tikzpicture}[scale=0.2,baseline=-3.5pt]%
\draw (0,0) node[inner sep=.5mm,minimum size=2mm](x1){\small$1$};
\draw (1,0) node[inner sep=.5mm,minimum size=2mm] (x2){\small$2$};
\draw (2,0) node[inner sep=.5mm,minimum size=2mm] (x3){\small$3$};
\draw[thick,looseness=1.]  (x1) to[out=90,in=90] (x2); 
\end{tikzpicture}};
\node at (-1.333,0.433){\begin{tikzpicture}[scale=0.2,baseline=-3.5pt]%
\draw (0,0) node[inner sep=.25mm,minimum size=2mm](x1){\small$1$};
\draw (1,0) node[inner sep=.25mm,minimum size=2mm] (x2){\small$2$};
\draw (2,0) node[inner sep=.25mm,minimum size=2mm] (x3){\small$3$};
\draw (0,0.3) node[inner sep=.5mm,minimum size=2mm](y1){};
\draw (2,0.3) node[inner sep=.5mm,minimum size=2mm] (y3){};
\draw[thick,looseness=1.]  (y1) to[out=90,in=90] (y3); 
\draw[thick,looseness=1.25]  (y1) to[out=105,in=75] (y3); 
\draw[thick,looseness=1.]  (x1) to[out=90,in=90] (x2); 
\end{tikzpicture}};
\node at (-0.5,0.233){\begin{tikzpicture}[scale=0.2,baseline=-3.5pt]%
\draw (0,0) node[inner sep=.5mm,minimum size=2mm](x1){\small$1$};
\draw (1,0) node[inner sep=.5mm,minimum size=2mm] (x2){\small$2$};
\draw (2,0) node[inner sep=.5mm,minimum size=2mm] (x3){\small$3$};
\draw[thick,looseness=.75]  (x1) to[out=90,in=90] (x3); 
\end{tikzpicture}};
\node at (0.,0.633){\begin{tikzpicture}[scale=0.2,baseline=-3.5pt]%
\draw (0,0) node[inner sep=.5mm,minimum size=2mm](x1){\small$1$};
\draw (1,0) node[inner sep=.5mm,minimum size=2mm] (x3){\small$3$};
\draw (2,0) node[inner sep=.5mm,minimum size=2mm] (x2){\small$2$};
\draw[thick,looseness=.75]  (x1) to[out=90,in=90] (x2); 
\end{tikzpicture}};
\node at (1.5,0.433){\begin{tikzpicture}[scale=0.2,baseline=-3.5pt]%
\draw (0,0) node[inner sep=.5mm,minimum size=2mm](x3){\small$3$};
\draw (1,0) node[inner sep=.5mm,minimum size=2mm] (x1){\small$1$};
\draw (2,0) node[inner sep=.5mm,minimum size=2mm] (x2){\small$2$};
\draw[thick,looseness=1.]  (x1) to[out=90,in=90] (x2); 
\end{tikzpicture}};
\node at (-2.7,-0.433){\begin{tikzpicture}[scale=0.2,baseline=-3.5pt]%
\draw (0,0) node[inner sep=.5mm,minimum size=2mm](x2){\small$2$};
\draw (1,0) node[inner sep=.5mm,minimum size=2mm] (x1){\small$1$};
\draw (2,0) node[inner sep=.5mm,minimum size=2mm] (x3){\small$3$};
\end{tikzpicture}};
\node at (-2.,-0.866){\begin{tikzpicture}[scale=0.2,baseline=-3.5pt]%
\draw (0,0) node[inner sep=.5mm,minimum size=2mm](x2){\small$2$};
\draw (1,0) node[inner sep=.5mm,minimum size=2mm] (x1){\small$1$};
\draw (2,0) node[inner sep=.5mm,minimum size=2mm] (x3){\small$3$};
\draw[thick,looseness=1.5]  (x1) to[out=90,in=90] (x3); 
\draw[thick,looseness=1.]  (x1) to[out=80,in=100] (x3); 
\end{tikzpicture}};
\node at (-1.,-0.866){\begin{tikzpicture}[scale=0.2,baseline=-3.5pt]%
\draw (0,0) node[inner sep=.5mm,minimum size=2mm](x2){\small$2$};
\draw (1,0) node[inner sep=.5mm,minimum size=2mm] (x1){\small$1$};
\draw (2,0) node[inner sep=.5mm,minimum size=2mm] (x3){\small$3$};
\draw[thick,looseness=1.]  (x1) to[out=90,in=90] (x3); 
\end{tikzpicture}};
\node at (0.,-0.866){\begin{tikzpicture}[scale=0.2,baseline=-3.5pt]%
\draw (0,0) node[inner sep=.5mm,minimum size=2mm](x2){\small$2$};
\draw (1,0) node[inner sep=.5mm,minimum size=2mm] (x3){\small$3$};
\draw (2,0) node[inner sep=.5mm,minimum size=2mm] (x1){\small$1$};
\end{tikzpicture}};
\node at (1.5,-0.866){\small$321$};
\end{tikzpicture}\hfill\vrule\hfill
\begin{tikzpicture}[scale=1.1]
\draw  [very thick]  (-0.866, -1.5) -- (1.732, 3.)  node [pos=1.,above,sloped]  {\scriptsize$x=z$};
\draw  [very thick]  (-1.732, 3.) -- (0.866, -1.5)  node [pos=1.05,above,sloped]  {\scriptsize$y=z$};
\draw  [very thick]  (-3., 0.) -- (2., 0.) node [pos=1.05,above,sloped]  {\scriptsize$x=y$};
\draw  [thick,dashed]  (-2.866, -1.5) -- (-0.268, 3.)  node [pos=1.1,above,sloped]  {\scriptsize$x=z+2$};
\draw  [thick,dashed]  (-1.866, -1.5) -- (0.732,  3.)  node [pos=1.1,above,sloped]  {\scriptsize$x=z+1$};
\draw  [thick,dashed]  (-3., 0.866) -- (2., 0.866) node [pos=1.1,above,sloped]  {\scriptsize$x=y+1$};
\node at (-1.,2.533) {\small$133$};
\node at (0.,2.533){\small$132$};
\node at (1.,2.533){\small$131$};
\node at (-2.,1.433){\small$123$};
\node at (1.5,1.433){\small$231$};
\node at (-1.,1.1){\small$122$};
\node at (-2.7,0.433){\small$113$};
\node at (-1.333,0.433){\small$112$};
\node at (-0.5,0.233){\small$111$};
\node at (0.,0.633){\small$121$};
\node at (1.5,0.433){\small$221$};
\node at (-2.7,-0.433){\small$213$};
\node at (-2.,-0.866){\small$212$};
\node at (-1.,-0.866){\small$211$};
\node at (0.,-0.866){\small$311$};
\node at (1.5,-0.866){\small$321$};
\end{tikzpicture}\hfill\strut
\caption{Pak-Stanley labelling of $\Ish_3$}
\label{isharr}
\end{figure}
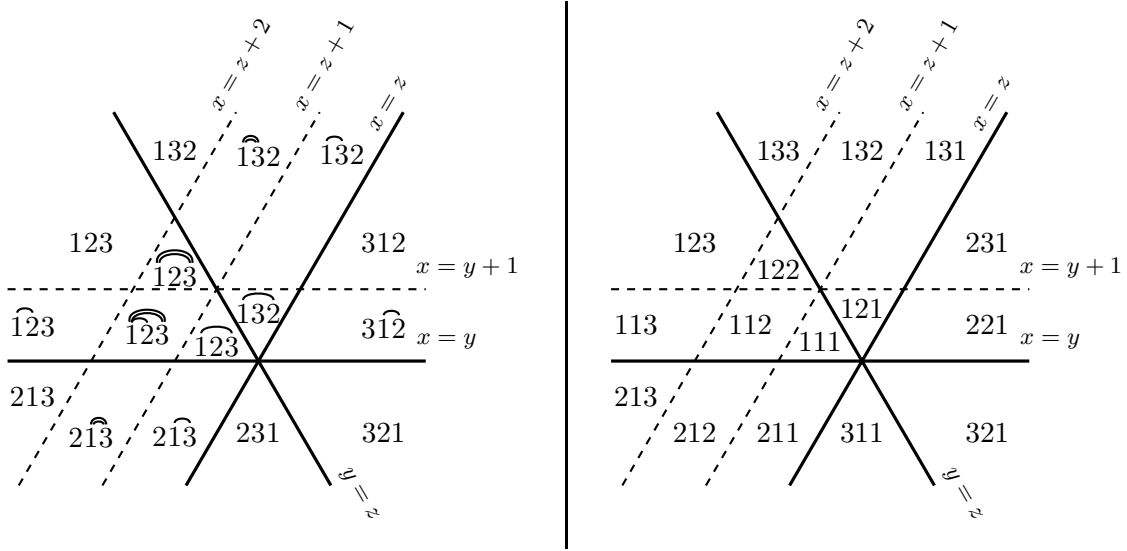

Then it is not difficult to directly find the label of a given region (cf. Stanley \cite{Stan2} in the case where $\cA$ is the Shi arrangement). Let  again $\cR$ be defined as in \eqref{eqr}
and
\renewcommand{\theenumi}{\ref{def.PSl}.\arabic{enumi}}
\renewcommand{\labelenumi}{\theenumi. }
\begin{enumerate}
\item\label{PSl1} \textbf{take $\bt=\bt(\bw)=(t_1,\dotsc,t_n)$} where
$\  t_{w_i}=\big|\big\{j\leq i\mid w_j\geq w_i\big\}\big|\ .$
\item\label{PSl2} \textbf{add $(a_{ij}-1)\be_j$ to $\bt$}  for every  hyperplane $(i,j,a_{ij})\in\mathcal{H}$.
\item\label{PSl3} \textbf{add $m_{ij}\be_j$ to $\bt$} for every pair $(i,j)$ with $1\leq i<j\leq n$ and $x_i>x_j$
such that $(i,j,a)\notin\mathcal{H}$ for every $a\in\N$.
\end{enumerate}
\renewcommand{\theenumi}{\arabic{enumi}}
\renewcommand{\labelenumi}{\theenumi. }
In fact, 
$\bt(\bw)$ is the  label of the region of the \emph{Coxeter arrangement} \footnote{I.e., the arrangement  $\big\{C_{ij} \mid 1\leq i<j\leq n \big\}$.} (cf. \cite[ad.]{Stan})
$$\cR'=\big\{(x_1,\dotsc,x_n)\in\Reals^n\bigm| x_{w_1}>x_{w_2}>\dotsb>x_{w_n} \big\}$$
on the Pak-Stanley labelling, and is also the label of the (unique) region of $\cA$ contained in $\cR'$ adjacent to the line defined by $x_1=\dots=x_n$.
Clearly, this region is represented by the permutation $w_1\dotsb w_n$, where all pairs $(i,j)$ such that
$1\leq i<j\leq n$  and such that there exists in $\cA$ a hyperplane of equation $x_i-x_j=a_{ij}>0$ are covered by a single arc.
For example, for every integer $n\geq 2$ and every $2\leq k\leq n$,
$\ell(\cA^k_n,\cR_0)=
\begin{tikzpicture}[scale=.2,baseline=-3.5pt];
\draw (0,0) node[inner sep=.5mm,minimum size=2mm](x1){$1$};
\draw (1,0) node[inner sep=.5mm,minimum size=1mm] (x2){$2$};
\draw (2,0) node[inner sep=.5mm,minimum size=2mm] (xb1){${}\cdot{}$};
\draw (2.5,0) node[inner sep=.5mm,minimum size=2mm] (xb2){${}\cdot{}$};
\draw (3,0) node[inner sep=.5mm,minimum size=2mm] (xb3){${}\cdot{}$};
\draw (4,0) node[inner sep=.5mm,minimum size=1mm] (xn){\raisebox{-7.5pt}{$n$}};
\draw[thick,looseness=.5]  (x1) to[out=90,in=90] (xn) ;
\end{tikzpicture}.$

For every hyperplane that is crossed, either the color of the arc connecting $i$ and $j$ is increased by one or, if the color is already as high as possible, the arc disappears.
Hence,  e.g. the region separated of $\mathcal{R}_0$ by the hyperplane of equation $x_1=x_n+1$ is
represented by
 \begin{tikzpicture}[scale=.225,baseline=-3.5pt];
\draw (0,0) node[inner sep=.5mm,minimum size=2mm](x1){$1$};
\draw (1,0) node[inner sep=.5mm,minimum size=1mm] (x2){$2$};
\draw (2,0) node[inner sep=.5mm,minimum size=2mm] (xb1){${}\cdot{}$};
\draw (2.5,0) node[inner sep=.5mm,minimum size=2mm] (xb2){${}\cdot{}$};
\draw (3,0) node[inner sep=.5mm,minimum size=2mm] (xb3){${}\cdot{}$};
\draw (6,0) node[inner sep=.25mm,minimum size=1mm] (xnm){\raisebox{-7.5pt}{\small$(\text{\normalsize$n\!\!-\!\!1$})$}};
\draw (9,0) node[inner sep=.25mm,minimum size=1mm] (xn){\raisebox{-7.5pt}{$n$}};
\draw[thick,looseness=.35]  (x1) to[out=90,in=90] (xnm) ;
\draw[thick,looseness=.475]  (x1) to[out=90,in=90] (xn) ;
\draw[thick,looseness=.4]  (x1) to[out=80,in=100] (xn) ;
\end{tikzpicture} and its Pak-Stanley label is $1\dotsb12$.

Note that our representation in the Ish case, since $1$ is the initial point of all arcs, is equivalent to the representation already given by Armstrong and Rhoades \cite{AR} and used by Leven, Rhoades and Wilson \cite{LRW}.

For another example, let $n=4$ and consider the region in $\cA^k_4$ of label $2311$ which is adjacent to the line defined by $x_1=x_2=x_3=x_4$ and contained in
$$\cR'=\big\{(x_1,x_2,x_3,x_4)\in\Reals^4\mid x_3>x_1>x_4>x_2\big\}\,.$$
This region is represented by
\begin{tikzpicture}[scale=.2,baseline=-3.5pt];
\draw (0,0) node[inner sep=.5mm,minimum size=2mm](x3){$3$};
\draw (1,0) node[inner sep=.5mm,minimum size=1mm] (x1){$1$};
\draw (2,0) node[inner sep=.5mm,minimum size=2mm] (x4){$4$};
\draw (3,0) node[inner sep=.5mm,minimum size=2mm] (x2){$2$};
\draw[thick,looseness=1.]  (x3) to[out=90,in=90] (x4) ;
\draw[thick,looseness=1.]  (x1) to[out=90,in=90] (x2) ;
\end{tikzpicture}
in $\Shi_4=\cA^2_4$ and in $\cA^3_4$, and by
\begin{tikzpicture}[scale=.2,baseline=-3.5pt];
\draw (0,0) node[inner sep=.5mm,minimum size=2mm](x3){$3$};
\draw (1,0) node[inner sep=.5mm,minimum size=1mm] (x1){$1$};
\draw (2,0) node[inner sep=.5mm,minimum size=2mm] (x4){$4$};
\draw (3,0) node[inner sep=.5mm,minimum size=2mm] (x2){$2$};
\draw[thick,looseness=1.]  (x1) to[out=90,in=90] (x2) ;
\end{tikzpicture} in $\Ish_4=\cA^4_4$.
In all the three cases,  there are five regions contained in $\cR'$ which are described in Table~\ref{tbl}.

\begin{table}[ht]
\setlength{\tabcolsep}{4pt}
\begin{center}
\begin{tabular}{|r|c|c|c|c|c||r|c|c|c|c|c|}
\hline
\multicolumn{6}{|c||}{$\Shi_4=\cA^2_4\quad/\quad\cA^3_4$}&\multicolumn{6}{c|}{$\Ish_4=\cA^4_4$}\\
\hline
{\footnotesize\textbf{Region}}&
\begin{tikzpicture}[scale=.2,baseline=-3.5pt];
\draw (0,0) node[inner sep=.5mm,minimum size=2mm](x3){$3$};
\draw (1,0) node[inner sep=.5mm,minimum size=1mm] (x1){$1$};
\draw (2,0) node[inner sep=.5mm,minimum size=2mm] (x4){$4$};
\draw (3,0) node[inner sep=.5mm,minimum size=2mm] (x2){$2$};
\draw[thick,looseness=.75]  (x3) to[out=90,in=90] (x4) ;
\draw[thick,looseness=.75]  (x1) to[out=90,in=90] (x2) ;
\end{tikzpicture}&
\begin{tikzpicture}[scale=.2,baseline=-3.5pt];
\draw (0,0) node[inner sep=.5mm,minimum size=2mm](x3){$3$};
\draw (1,0) node[inner sep=.5mm,minimum size=1mm] (x1){$1$};
\draw (2,0) node[inner sep=.5mm,minimum size=2mm] (x4){$4$};
\draw (3,0) node[inner sep=.5mm,minimum size=2mm] (x2){$2$};
\draw[thick,looseness=.75]  (x1) to[out=90,in=90] (x2) ;
\end{tikzpicture}&
\begin{tikzpicture}[scale=.2,baseline=-3.5pt];
\draw (0,0) node[inner sep=.5mm,minimum size=2mm](x3){$3$};
\draw (1,0) node[inner sep=.5mm,minimum size=1mm] (x1){$1$};
\draw (2,0) node[inner sep=.5mm,minimum size=2mm] (x4){$4$};
\draw (3,0) node[inner sep=.5mm,minimum size=2mm] (x2){$2$};
\draw[thick,looseness=.75]  (x3) to[out=90,in=90] (x4) ;
\end{tikzpicture}&
\begin{tikzpicture}[scale=.2,baseline=-3.5pt];
\draw (0,0) node[inner sep=.5mm,minimum size=2mm](x3){$3$};
\draw (1,0) node[inner sep=.5mm,minimum size=1mm] (x1){$1$};
\draw (2,0) node[inner sep=.5mm,minimum size=2mm] (x4){$4$};
\draw (3,0) node[inner sep=.5mm,minimum size=2mm] (x2){$2$};
\draw[thick,looseness=1.]  (x1) to[out=90,in=90] (x4) ;
\end{tikzpicture}&
\begin{tikzpicture}[scale=.2,baseline=-3.5pt];
\draw (0,0) node[inner sep=.5mm,minimum size=2mm](x3){$3$};
\draw (1,0) node[inner sep=.5mm,minimum size=1mm] (x1){$1$};
\draw (2,0) node[inner sep=.5mm,minimum size=2mm] (x4){$4$};
\draw (3,0) node[inner sep=.5mm,minimum size=2mm] (x2){$2$};
\end{tikzpicture}&
{\footnotesize\textbf{Region}}&
\begin{tikzpicture}[scale=.2,baseline=-3.5pt];
\draw (0,0) node[inner sep=.5mm,minimum size=2mm](x3){$3$};
\draw (1,0) node[inner sep=.5mm,minimum size=1mm] (x1){$1$};
\draw (2,0) node[inner sep=.5mm,minimum size=2mm] (x4){$4$};
\draw (3,0) node[inner sep=.5mm,minimum size=2mm] (x2){$2$};
\draw[thick,looseness=.75]  (x1) to[out=90,in=90] (x2) ;
\end{tikzpicture}&
\begin{tikzpicture}[scale=.2,baseline=-3.5pt];
\draw (0,0) node[inner sep=.5mm,minimum size=2mm](x3){$3$};
\draw (1,0) node[inner sep=.5mm,minimum size=1mm] (x1){$1$};
\draw (2,0) node[inner sep=.5mm,minimum size=2mm] (x4){$4$};
\draw (3,0) node[inner sep=.5mm,minimum size=2mm] (x2){$2$};
\draw[thick,looseness=1.]  (x1) to[out=90,in=90] (x4) ;
\end{tikzpicture}&
\begin{tikzpicture}[scale=.2,baseline=-3.5pt];
\draw (0,0) node[inner sep=.5mm,minimum size=2mm](x3){$3$};
\draw (1,0) node[inner sep=.5mm,minimum size=1mm] (x1){$1$};
\draw (2,0) node[inner sep=.5mm,minimum size=2mm] (x4){$4$};
\draw (3,0) node[inner sep=.5mm,minimum size=2mm] (x2){$2$};
\draw[thick,looseness=1.1]  (x1) to[out=90,in=90] (x4) ;
\draw[thick,looseness=1.4]  (x1) to[out=100,in=80] (x4) ;
\end{tikzpicture}&
\begin{tikzpicture}[scale=.2,baseline=-3.5pt];
\draw (0,0) node[inner sep=.5mm,minimum size=2mm](x3){$3$};
\draw (1,0) node[inner sep=.5mm,minimum size=1mm] (x1){$1$};
\draw (2,0) node[inner sep=.5mm,minimum size=2mm] (x4){$4$};
\draw (3,0) node[inner sep=.5mm,minimum size=2mm] (x2){$2$};
\draw[thick,looseness=.7]  (x1) to[out=90,in=90] (x4) ;
\draw[thick,looseness=1.1]  (x1) to[out=100,in=80] (x4) ;
\draw[thick,looseness=1.4]  (x1) to[out=110,in=70] (x4) ;
\end{tikzpicture}&
\begin{tikzpicture}[scale=.2,baseline=-3.5pt];
\draw (0,0) node[inner sep=.5mm,minimum size=2mm](x3){$3$};
\draw (1,0) node[inner sep=.5mm,minimum size=1mm] (x1){$1$};
\draw (2,0) node[inner sep=.5mm,minimum size=2mm] (x4){$4$};
\draw (3,0) node[inner sep=.5mm,minimum size=2mm] (x2){$2$};
\end{tikzpicture}\\
{\footnotesize\textbf{Label}}&$2311$&$2312$&$2411$&$2412$&$2413$&
{\footnotesize\textbf{Label}}&$2311$&$2411$&$2412$&$2413$&$2414$\\
\hline
\end{tabular}
\end{center}
\medskip
\caption{Labels in $\cA^k_4$ of the regions whose points satisfy $x_3>x_1>x_4>x_2$.}
\label{tbl}
\end{table}
Note that in all three arrangements the regions labelled $2411$ are separated from the region labelled $2311$ by the hyperplane of
equation $x_1-x_2=1$. The first label is given by \eqref{PSl1} and the second one by \eqref{PSl3}.
Now, the regions labelled $2411$ and $2412$ on the left-hand side of the table are separated from each other by the hyperplane of equation $x_3-x_4=1$, whereas the latter is separated from the region labelled $2413$ by the hyperplane of equation $x_1-x_4=1$.
Hence, $2412$ and $2413$ are also labels given by \eqref{PSl3}.
The regions labelled $2411$, $2412$, $2413$ and $2414$ on the right-hand side of the table are separated from one another by the hyperplane of equation $x_1-x_4=a$, where $a=1$ and $a=2$, and where $a=3$, respectively.
The first two labels, $2412$ and $2413$, are given by \eqref{PSl2} and the last one, $2414$,
by \eqref{PSl3}.

Finally, note that in $\Ish_4$ the region labelled by $2312$ is not contained in $\cR'$. In fact,
$2312=2211+0101=
\ell\big(\cA^4_4,\begin{tikzpicture}[scale=.2,baseline=-3.5pt];
\draw (0,0) node[inner sep=.5mm,minimum size=2mm](x3){$3$};
\draw (1,0) node[inner sep=.5mm,minimum size=2mm] (x1){$1$};
\draw (2,0) node[inner sep=.5mm,minimum size=2mm] (x2){$2$};
\draw (3,0) node[inner sep=.5mm,minimum size=2mm] (x4){$4$};
\draw[thick,looseness=.75]  (x1) to[out=90,in=90] (x4);
\draw[thick,looseness=.95]  (x1) to[out=100,in=80] (x4);
\end{tikzpicture}\big)$
since we have $\mathcal{H}=$ {\small$\big\{(1,4,2)\big\}$ for the region
\begin{tikzpicture}[scale=.2,baseline=-3.5pt];
\draw (0,0) node[inner sep=.5mm,minimum size=2mm](x3){$3$};
\draw (1,0) node[inner sep=.5mm,minimum size=2mm] (x1){$1$};
\draw (2,0) node[inner sep=.5mm,minimum size=2mm] (x2){$2$};
\draw (3,0) node[inner sep=.5mm,minimum size=2mm] (x4){$4$};
\draw[thick,looseness=.75]  (x1) to[out=90,in=90] (x4);
\draw[thick,looseness=.95]  (x1) to[out=100,in=80] (x4);
\end{tikzpicture} of $\cA^4_4$
and hence $(1,2,a)\notin\mathcal{H}$ for every $a\in\N$ --- although in this region $x_1>x_2$.
In both the remaining arrangements, $\cA^2_4$ and $\cA^3_4$,
$\mathcal{H}=$ {\small$\big\{(1,2,1),(1,4,1)\big\}$} for the region
$\begin{tikzpicture}[scale=.2,baseline=-3.5pt];
\draw (0,0) node[inner sep=.5mm,minimum size=2mm](x3){$3$};
\draw (1,0) node[inner sep=.5mm,minimum size=2mm] (x1){$1$};
\draw (2,0) node[inner sep=.5mm,minimum size=2mm] (x4){$4$};
\draw (3,0) node[inner sep=.5mm,minimum size=2mm] (x2){$2$};
\draw[thick,looseness=.75]  (x1) to[out=90,in=90] (x2);
\end{tikzpicture}$, and hence
$(3,4,a)\notin\mathcal{H}$ for every $a\in\N$. Yet, the hyperplane of equation
$x_3-x_4=1$ belongs to both arrangements.}
\footnote{Note that
$\ell\big(\cA^3_4,\begin{tikzpicture}[scale=.175,baseline=-3.5pt];
\draw (0,0) node[inner sep=.5mm,minimum size=2mm](x3){$3$};
\draw (1,0) node[inner sep=.5mm,minimum size=2mm] (x1){$1$};
\draw (2,0) node[inner sep=.5mm,minimum size=2mm] (x2){$2$};
\draw (3,0) node[inner sep=.5mm,minimum size=2mm] (x4){$4$};
\draw[thick,looseness=.7]  (x1) to[out=90,in=90] (x4);
\draw[thick,looseness=1.]  (x1) to[out=100,in=80] (x4);
\end{tikzpicture}\big)=231\mathbf{3}$.}

\smallskip
In the right-hand side of Figure~\ref{isharr} the Pak-Stanley labelling of the regions of $\Ish_3$ is shown.
In dimension $n$, these labels form the set of \emph{$n$-dimensional Ish-parking functions}, 
characterized in a previous article \cite{DGO3}.
The labels of the regions of $\Shi_n$ form the set of \emph{$n$-dimensional parking functions}, defined below, as proven by Pak and Stanley in their seminal work \cite{Stan}. 

Parking functions and Ish-parking functions, as well as the Pak-Stanley labels of $\cA^k_n$ for $2<k<n$, are   \emph{graphical parking functions} 
as introduced by Postnikov and Shapiro \cite{PosSha} and  reformulated by Mazin \cite{Mazin}.

\section{Graphical parking functions}
\begin{definition} [\cite{Mazin}, ad.]
\label{def-GPark}
Let $\cG=(V,A)$ be a (finite) directed loopless connected multigraph, where $V=[n]$ 
for some natural  $n$. Then
$\ba=(a_1,\dotsc,a_n)\in\N^n$ is a \emph{$\cG$-parking function} if for every non-empty subset $I\subseteq[n]$ there exists a vertex $i\in I$ such that the number of arcs $(i,j)\in A$ with $j\notin I$, counted with multiplicity, is greater than $a_i-2$.
\end{definition}

Given the arrangement $\cA^k_n$, consider a multigraph $\cG^k_n$ where for each hyperplane of equation $x_i=x_j$ there is a corresponding arc $(i,j)$, and for each hyperplane of equation $x_i=x_j+a$ with $a\in\N$ there is a corresponding
arc $(j,i)$. In Figure~\ref{fig1}, the graphs $\cG^2_4$, $\cG^3_4$ and $\cG^4_4$ are shown.
Note that $\cG^2_n$ is the complete digraph $K_n$ on $n$ vertices.
We will use the following crucial result.

\begin{theorem}[Mazin \cite{Mazin}, ad.]
For every $2\leq k\leq n$,  the set
\begin{equation*}
\big\{\ell(\cA^k_n,\cR)\mid \text{ $\cR$ is a region of $\cA^k_n$}\big\} \end{equation*}
 is the set of $\cG^k_n$-parking functions. 
\end{theorem}

\begin{figure}[t]
\noindent
\strut\hfill
\begin{tikzpicture}[scale=1.75]
\draw  [directed, thick]  (0,-1) to [out=75,in=295] (0,0);
\draw  [directed]  (0,0) -- (0,-1);
\draw  [directed]  (0,0) -- (.86,.5);
\draw  [directed]  (0,0) -- (-.86,.5);
\draw  [directed, thick]  (.86,.5) to [out=195,in=55] (0,0);
\draw  [directed, thick]  (-.86,.5) to [out=315,in=170] (0,0);
\draw  [directed]  (.86,.5) -- (0,-1);
\draw [fill] (0,0) node [above] {$1$} circle [radius=.0425];
\draw [fill] (0,-1) node [below] {$4$} circle [radius=.0425];
\draw  [directed]  (-.86,.5) -- (0,-1) ;
\draw [fill] (-.86,.5) node [above left] {$2$} circle [radius=.0425];
\draw  [directed]  (-.86,.5) -- (.86,.5) ;
\draw [fill] (.86,.5) node [above right] {$3$} circle [radius=.0425];
\draw  [directed, thick]  (.86,.5) to [out=195,in=345] (-.86,.5);
\draw  [directed, thick]  (0,-1) to [out=75,in=225] (.86,.5);
\draw  [directed, thick]  (0,-1) to [out=100,in=315] (-.86,.5);
\end{tikzpicture}
 \hfill\vrule\hfill
\begin{tikzpicture}[scale=1.75]
\draw  [directed, thick]  (0,-1) to [out=75,in=295] (0,0);
\draw  [directed]  (0,0) -- (0,-1);
\draw  [directed]  (0,0) -- (.86,.5);
\draw  [directed]  (0,0) -- (-.86,.5);
\draw  [directed, thick]  (.86,.5) to [out=195,in=55] (0,0);
\draw  [directed, thick]  (-.86,.5) to [out=315,in=170] (0,0);
\draw  [directed, thick]  (.86,.5) to [out=185,in=150]  (-.08,.0525) -- (0,0);
\draw  [directed]  (.86,.5) -- (0,-1);
\draw [fill] (0,0) node [above] {$1$} circle [radius=.0425];
\draw [fill] (0,-1) node [below] {$4$} circle [radius=.0425];
\draw  [directed]  (-.86,.5) -- (0,-1) ;
\draw [fill] (-.86,.5) node [above left] {$2$} circle [radius=.0425];
\draw  [directed]  (-.86,.5) -- (.86,.5) ;
\draw [fill] (.86,.5) node [above right] {$3$} circle [radius=.0425];
\draw  [directed, thick]  (0,-1) to [out=75,in=225] (.86,.5);
\draw  [directed, thick]  (0,-1) to [out=105,in=170] (0,0);
\end{tikzpicture}
\hfill\vrule\hfill
\begin{tikzpicture}[scale=1.75]
\draw  [directed, thick]  (0,-1) to [out=75,in=295] (0,0);
\draw  [directed]  (0,0) -- (0,-1);
\draw  [directed]  (0,0) -- (.86,.5);
\draw  [directed]  (0,0) -- (-.86,.5);
\draw  [directed, thick]  (.86,.5) to [out=195,in=55] (0,0);
\draw  [directed, thick]  (.86,.5) to [out=185,in=150]  (-.08,.0525) -- (0,0);
\draw  [directed, thick]  (-.86,.5) to [out=315,in=170] (0,0);
\draw  [directed]  (.86,.5) -- (0,-1);
\draw [fill] (0,0) node [above] {$1$} circle [radius=.0425];
\draw [fill] (0,-1) node [below] {$4$} circle [radius=.0425];
\draw  [directed]  (-.86,.5) -- (0,-1) ;
\draw [fill] (-.86,.5) node [above left] {$2$} circle [radius=.0425];
\draw  [directed]  (-.86,.5) -- (.86,.5) ;
\draw [fill] (.86,.5) node [above right] {$3$} circle [radius=.0425];
\draw  [directed, thick]  (0,-1) to [out=65,in=10] (0,0);
\draw  [directed, thick]  (0,-1) to [out=105,in=170] (0,0);
\end{tikzpicture}
\hfill\strut

\strut\hspace{1.875cm}$\cG^2_4=K_4$\hfill$\cG^3_4$\quad\hfill$\cG_4^4$\hspace{2.5cm}\strut
\caption{Directed multi-graphs associated with $\Shi_4=\cA_4^2$, with $\cA_4^3$  and with $\Ish_4=\cA_4^4$.}\label{fig1}
\end{figure}
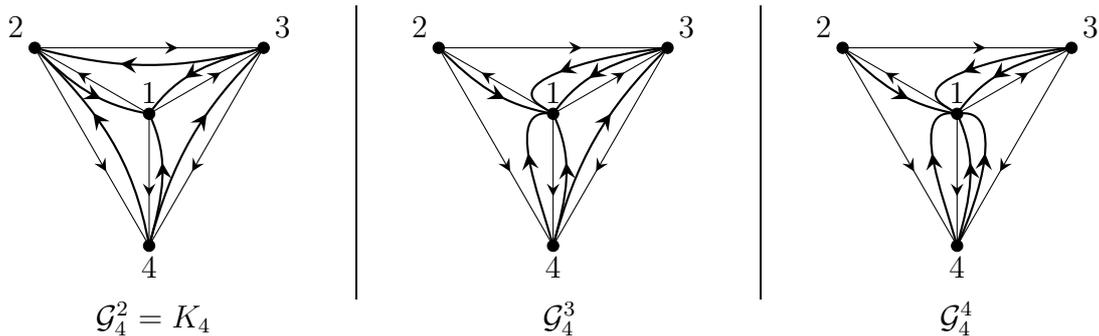

\subsection{Parking functions}
\begin{definition}\label{def31}
The $n$-tuple $\ba=(a_1,\dotsc,a_n)\in[n]^n$ is an $n$-dimensional parking function if
\footnote{With this definition, $\mathbf{1}:=(1,\dotsc,1)\in[n]^n$ \emph{is} a parking function and
$\mathbf{0}:=(0,\dotsc,0)\in[n]^n$ \emph{is not}. Parking functions are sometimes defined differently, so as to contain $\mathbf{0}$ (and not $\mathbf{1}$).
In that case, they are the elements of form $\bb=\ba-\mathbf{1}$ for  $\ba$ a parking function in the current sense.} 
\begin{equation*}\big|\big\{j\in[n]\mid a_j\leq i\big\}\big|\geq i\,,\quad \forall i\in[n]\,.\end{equation*}
\end{definition}

Note that parking functions (sometimes called \emph{classical parking functions}) are indeed $\cG^2_n$-parking functions, being  $\cG^2_n=K_n$, the complete digraph on $[n]$. In fact, suppose that $\ba$ is a $K_n$-parking function.
Then, given $i \in [n]$, let $I = $ {\small$\{ j \in [n] \mid a_j > i \}$}.
If $I = \varnothing$, then $| \{ j \in [n] \mid a_j \leq i \} | = n \geq i$.
If $I \neq \varnothing$, then there is $\ell \in I$ such that $| ${\small $\{ (\ell,j) \in A \mid j \notin I \}$}$|\geq a_\ell-1$ and so
$$ | \{ j \in [n] \mid a_j \leq i \} | = | \{ (\ell,j) \in A \mid j \notin I \} | \geq a_\ell-1 \geq i\,,$$
the last inequality since ${\ell\in I}$. The other direction is obvious.

Konheim and Weiss \cite{KW} introduced the concept of parking functions that can be thus described. Suppose  that $n$ drivers want to park in a one-way street with exactly $n$ places
and that $\ba\in[n]^n$ is the record of the preferred parking slots, that is, $a_i$ is the preferred parking place of driver $i\in[n]$. They enter the street one by one, driver $i$ immediately after driver $i-1$ parks, directly looks after his/her favourite slot, and if it is occupied he/she tries to park in the first free slot thereafter --- or leaves the street if no one exists.
Konheim and Weiss showed that $\ba$ is a parking function if and only if  all the drivers can park in the street in this way. 

In other words, consider the following  algorithm.

{\small
\begin{algorithm}[h]
\begin{quote}
\begin{algorithmic}[1]
  \Statex
  \Statex{\hspace{-0.5cm}\sc Parking Algorithm}
  \Statex{\bf Input:}\  $\ba\in[n]^n$
  \State $\spp =(0,\dotsc,0)\in\Z^{2n}$
      \ForAll{$i\in[n]$ in \emph{descending} order}
          \State $p=a_i$
          \While{$\spp(p) \neq 0$}
          \State increase $p$
          \EndWhile
          \State $\pp(i)=p$.
          \State $\spp(p)=i$
      \EndFor
  \Statex{\bf Output:}\  {$\spp$, $\pp$}
\end{algorithmic}
\end{quote}
\end{algorithm}

We say that \emph{$\ba$ parks $i\in[n]$} if 
 $\pp(i)\leq n$.
Parking functions are those which park every element,
or, equivalently, if we set
\begin{align*}
&\ff:=\min\{i\in[n+1]\mid \spp(i)=0\}\quad\text{and}\\
&\op=\spp^{-1}([n])\,,
\end{align*}
those for which $\ff=n+1$ or those for which $\op=[n]$.

Note that by Definition~\ref{def31} $\mathfrak{S}_n$ acts on the set $\PF_n$ of size $n$ parking functions: if
$\bw\in\mathfrak{S}_n$ and $\bw(\ba):=\ba\circ\bw=(a_{w_1},\dotsc,a_{w_n})$, then $\ba\in\PF_n$ if and only if $\bw(\ba)\in\PF_n$. 
In fact, this is a particular case of a more general situation, described in the following result.
\begin{lemma}\label{shift}
Given $\ba\in[n]^n$ and $\bw\in\mathfrak{S}_n$,
$$\op(\ba)=\op(\ba\circ\bw)\,.$$
\end{lemma}
\begin{proof}
It is sufficient to prove the claim when $\bw$ is the transposition $(i\,i+1)$ for some $i\in[n-1]$. Let
$\bb:=\ba\circ\bw=(b_1,\dotsc,b_n)=(a_1,\dotsc,a_{i-1},a_{i+1},a_i,a_{i+2},\dotsc,a_n)$, 
$\alpha:=\pp(i+1)\geq a_{i+1}$ and $\beta:=\pp(i)\geq a_i$ when  the Parking Algorithm is applied to $\ba$.

Suppose that $\beta<\alpha$. Then, since $a_i\leq\beta$, 
 $\beta=\pp(i)$ and $\alpha=\pp(i+1)$ when the algorithm is applied to $\bb$.
Now, suppose that $\alpha<\beta$. Hence,  if $b_{i+1}\,(=\!a_i)>\alpha$, then $\beta=\pp(i+1)$ and $\alpha=\pp(i)$ when the algorithm is applied to $\bb$, and if  $b_{i+1}\leq\alpha$, then $\alpha=\pp(i+1)$ and $\beta=\pp(i)$.
\end{proof}

\subsection{Ish-parking functions}
The labels of the regions of $\Ish_n$,  the \emph{Ish-parking functions},  are characterized as follows. 
\begin{definition} 
Let $\ba=(a_1,\dotsc,a_n)\in\N^n$ and $1<m\leq n$.
The \emph{centre}  of $\ba$,
$Z(\ba)$, is the (possibly empty) largest set $Z = \{ i_1, \ldots, i_m \}$ contained in $[n]$  with $\ n\geq i_1 > \cdots > i_m\geq1\ $ 
and the property
\footnote{Note that if this property holds for both  $X,Y\subseteq[n]$ then it holds for $X\cup Y$, and so this concept is well-defined  (cf. \cite{DGO,DGO2,DGO3}). The centre was previously called the \emph{reverse centre} \cite{DGO3}.} 
that $a_{i_j} \leq j$ for every $j \in [m]$.
\end{definition}
\begin{theorem}[{\cite[Proposition 3.12]{DGO3}}]\label{def.ipf}
The function 
$\ba\in[n]^n$ is an  \emph{Ish-parking function} if and only if $1\in Z(\ba)$.\hfill\qed
\end{theorem}

\begin{proposition}\label{parkscentre}
Any function $\ba\in[n]^n$ parks all the elements of $Z(\ba)$. Moreover,
for every $\bb\in[n]^n$, if the restriction to $Z(\ba)$ of $\ba$ and $\bb$ are equal, then 
$\bb$ also parks all the elements of $Z(\ba)$.
\end{proposition}
\begin{proof}
Let $Z(\ba)=\{i_1,\dotsc,i_{m}\}$ with $i_1>\dotsb>i_m$. We show that  if $a_{i_j}\leq j$ for every $j=1,\dotsc,m$ then $\ba$ parks all the elements of $Z(\ba)$. In fact, it is immediate to see by induction on $j$ that when $p$ is assigned $a_{i_j}$ in Line~3 of the Parking Algorithm then $\spp(i)\neq0$ for every $i<p$, and the same happens if we replace $\ba$ with $\bb$ as described above, since $Z(\bb)\supseteq Z(\ba)$. Hence,
$\ff(\ba)>p=a_{i_j}$ and $\ba$ parks $i_j$, and the same holds for $\bb$.
\end{proof}

\subsection{Partial parking functions}\label{ppf}

\noindent
Fixed integers $n\geq3$ and $1<k\leq n$, and $\ba=(a_1,\dotsc,a_n)\in\N^n$, 
consider $\pi\in\mathfrak{S}_n$ such that:
$$
\begin{cases}
\pi(i)=i,&\text{ for every $i< k$;}\\
a_{\pi(i)}\geq a_{\pi(i+1)},
&\text{for every $k\leq i < n$;}
\end{cases}$$
(note that if $k\leq i \leq n$ then also $k\leq\pi(i)\leq n$, since $\pi\in\mathfrak{S}_n$).
Finally, set
$$\bak:=\ba\circ\pi\,.$$

\begin{definition}\label{def37}
$\ba\in[n]^n$ is a \emph{$k$-partial parking function} if:
\begin{itemize}
\item $\ba$ parks all the elements of $[k,n]$;
\item $1\in Z(\bak)$.
\end{itemize}
\end{definition}

The restriction to $[k,n]$ of a function $\ba$ that  parks all the $n+1-k$ elements of $[k,n]$
is  a particular case of a \emph{defective parking function}  introduced by Cameron, Johannsen, Prellberg and Schweitzer \cite{CJPS}. Hence, the number $T_k$ of \emph{all} functions that park every element of $[k,n]$ is $n^{k-1}c(n,n+1-k,0)$, where $c(n,m,k)$ is the number of $(n,m,k)$-defective parking functions  \cite[pp.3]{CJPS}, that is
$$T_k=k\,n^{k-1}(n+1)^{n-k}\,.$$
\begin{lemma}\label{remrem}
A function \emph{$\ba\in[n]^n$ parks every element of $[k,n]$} if and only if
$$\big|\big\{j\in[k,n]\mid a_j\leq i\big\}\big| + k-1\geq i\,,\quad \forall i\in[k,n]\,.$$
\end{lemma}
\begin{proof}
In fact, since  this property does not depend on the first $k-1$ coordinates of $\ba$
we may replace each one of them by $1$. Now, the new function parks every element of $[k,n]$ if and only if it is a parking function.
\end{proof}

\begin{proposition}\label{def2.ppf}
A function $\ba=(a_1,\dotsc,a_n)\in[n]^n$ is a $k$-partial parking function if and only if
there is a permutation $\sigma\in\mathfrak{S}_n$ with 
\begin{equation*}
\begin{cases}
a_{\sigma(i)}\leq i \text{ for every } i\in[a_1]\text{ and for every }
i\in[k,n]\text{ such that }\sigma(i)\geq k\,;\\
\sigma(i+1)< \sigma(i) \text{ for every } i\in[a_1-1]\text{ such that }\sigma(i)<k\,.\end{cases}
\end{equation*}
\end{proposition}

\begin{proof}
Let $\ta_i=a_{\pi(i)}$ be the $i$th component of $\bak$ ($1\leq i\leq n$)
and $Z=Z(\bak)$. We suppose that, as in Definition~\ref{def.ipf},
$Z=\{ \alpha_1, \ldots, \alpha_z \}$  with $n\geq \alpha_1 > \cdots > \alpha_z\geq1$  and $\ta_{\alpha_i} \leq i$ for every 
$i \in [z]$. 
 Let $B=[k-1]\setminus Z=\{\beta_1,\dotsc,\beta_m\}$ and $C=[k,n]\setminus Z=\{\gammab{1},\dotsc,\gammab{\ell}\}$ with  $\beta_1<\dotsb<\beta_m$ and
 $\gammab{1}<\dotsb<\gammab{\ell}$.
Note that, in particular, $a_1\leq z$, $z+m+\ell=n$ and 
$\ta_{\gammad{1}}\geq\dotsb\geq\ta_{\gammad{\ell}}$.

Now, suppose that $\ba$ is a $k$-partial parking function as defined in Definition~\ref{def37}.
We define $\tau\in\mathfrak{S}_n$ by
$$\tau(t)=\begin{cases}
\alpha_t,& \text{ if $t\leq z$;}\\
\beta_{t-z},& \text{ if $z<t\leq n-\ell$;}\\
\gammab{n+1-t},& \text{ if $n-\ell<t\leq n$.}\\
\end{cases}$$
so that $\tau(1)>\dotsb>\tau(z)$  and $\ta_{\tau(n-\ell+1)}\leq\dotsb\leq\ta_{\tau(n)}$. 
Finally, we define $\sigma=\pi\circ\tau$.

Then, for every $i\in[a_1]\subseteq[z]$, $a_{\sigma(i)}=\ta_{\tau(i)}\leq i$   and, 
for every $i\in[a_1-1]$ such that $\tau(i)<k$,
$\tau(i+1)< \tau(i)=\sigma(i)$. But then $\tau(i)<k$ implies that $\tau(i+1)<k$ and thus $\sigma(i+1)=\tau(i+1)<\sigma(i)$.
Now, suppose that $i<\ta_{\tau(i)}$ for some $i\in[k,n]$  such that $\tau(i)\geq k$. Then $i\in C$.
Since   $\ta_{\tau(i)}\leq\ta_{\tau(j)}$ for every $i<j\leq n$ (being, in particular, also $j\in C$),
\begin{align*}
&\big|\big\{j\in[k,n]\mid \ta_j> i\big\}\big|>n-i\,,\\
\shortintertext{and thus, contrary to the fact that $\ba$ parks all the elements of $[k,n]$
(cf. Lemma~\ref{remrem}),}
&\big|\big\{j\in[k,n]\mid a_j\leq i\big\}\big|+k-1 < i\,.
\end{align*}
For example, suppose that $n=8$, $k=5$, and $\ba=2663\,1461$. Then $\bak=2663\,6411$,
$\tau=87412365$, $\sigma=85412367$ and $\bak\circ\tau=\ba\circ\sigma=1132\,6646$.

Conversely, suppose that $a_{\sigma(i)}\leq i$ for every $i\in[k,n]$  such that $\sigma(i)\geq k$. By definition of $\tau$, if $i\in[k,n]$ then
$i\in Z$ or $\tau(i)>z+m\geq k-1$. Therefore, $a_{\sigma(i)}=\ta_{\tau(i)}\leq i$ for every $i\in[k,n]$ and hence
$$\big|\big\{\ell\in[k,n]\mid a_{\ell}\leq j\big\}\big| +k-1\geq j\,$$
Finally, $\sigma([a_1])\cup\{1\}\subseteq Z$ by maximality of $Z$. 
\end{proof}

Indeed, $k$-partial parking functions are exactly the $\cG^k_n$-parking functions.
But to prove it we still need a different tool.

\section{The $\dfs$-Burning Algorithm}
We want to characterise the $\cG_n^k$-parking functions for every $k,n\in\N$ such that $2\leq k\leq n$.
Similarly to what we did for the characterisation of the Ish-parking functions \cite{DGO3} (the case $k=n$), our main tool is the $\dfs$-Burning Algorithm of  Perkinson, Yang and Yu \cite{PYY} (cf. Figure~\ref{algs}). Recall that this algorithm, given $\ba\in[n]^n$ and a multiple digraph $\cG$, determines whether $\ba$ is a $\cG$-parking function by constructing in the positive case an oriented spanning subtree $T$ of $\cG$ that is in bijection with $\ba$  \cite{PYY,DGO3}. The Tree to Parking Function Algorithm (cf. Figure~\ref{algs}, on the right) builds $\ba$ out of $T$ (and $\cG$), thus defining the inverse bijection.

Recall \cite{DGO3} that the algorithm is not directly applied to the multidigraph $\cG$. Indeed, it is applied to another digraph, $\bG$, with one more vertex, $0$, and set of arcs $\oA$ defined by:
\begin{itemize}
\item For every vertex $v\in[n]$, $(0,v)\in\oA$;
\item For every arc $(v,w)\in A$, $(w,v)\in\oA$.
\end{itemize}

We use the following result, which is an extension to directed multigraphs of the work of Perkinson, Yang and Yu \cite{PYY}.
\begin{proposition}[{\cite[Proposition 3.2]{DGO3}}]\label{propant}
Given a directed multigraph $\cG$ on $[n]$ and a function $\ba \colon [n] \to \N_0$,  $\ba$ is a $\cG$-parking function if and only if  the list $\bnv$ at the end of the execution of the $\dfs$-Burning Algorithm applied to
$\bG$ includes all the vertices in $\oV=\{0\}\cup[n]$.
\end{proposition}

The different arcs connecting $v$ and $w$ that occur $\ell$ times ($\ell>1$) are labelled $(w,v+m\,n)\in\oA$ with $m\in[0,\ell-1]$, so as to distinguish between them.
For this purpose, the $\dfs$-Burning Algorithm inputs the list $\cN(w)$   of vertices $v$ such that $(w,v)\in\bG$  for each vertex $w$  under the same form, that is, under the form 
$v+m\,n$ with $m\in[0,\ell-1]$. However, note that every vertex is seen by the algorithm as a unique entity. In fact, in Line~7 we take $j_n=\mathrm{Mod}(j,n)$ for every $j\in\cN(i)$ (in Line~6).

Note that although the order of the vertices in $\cN$ is not relevant in the context of Proposition~\ref{propant},  it is indeed relevant in other contexts, like that of Lemma~\ref{centro}  (cf. \cite[Remark 3.4.]{DGO3}).
 We define the order in $\bGk$ so that:
 \begin{enumerate}
 \item $\cN(0)$ if formed by the arcs of form $(0,i)$ for every $i\in[n]$; we sort $\cN(0)$ based on the value of $i$, in descending order.
 \item There is an arc of form  $(1,i+m\,n)$ for every $i>1$ and every $0\leq m\leq\min\{i,k\}-2$. We sort $\cN(1)$ by the value of $i$ in descending order, breaking ties by the value of $m$, again in descending order. For example, in
 $\cA^3_4$ (cf. Figure~\ref{examplex}),
 $$\cN(1)=\langle 8,4,7,3,2\rangle\,.$$
 \item For every $1\leq m<i$ there is a unique arc $(i,m)$. \emph{Exactly when $i\geq k$}, there is also an arc of form $(i,m)$ for every $i<m\leq n$. In all cases, we sort $\cN(i)$ by the value of $m$ in descending order
\end{enumerate} 
 
\begin{figure}[ht]
\begin{minipage}[t]{.45\textwidth}
{\footnotesize
\algrenewcommand\algorithmicindent{1.15em}%
\begin{algorithmic}[1]
  \Statex{\hspace{-0.5cm}\sc $\dfs$-Burning Algorithm (ad.)}
  \Statex{\bf Input:}\  $\ba\colon[n]\to\N$
  \State $\bnv =\{0\}$
  \State $\damp = \{\,\}$
  \State $\tree= \{\,\}$
  \State execute {\sc dfs\_from}($0$)
  \Statex{\bf Output:}\  {$\bnv$, $\tree$ and $\damp$}
  \Statex{\hspace{-0.5cm}\sc auxiliary function}
    \Function{\sc dfs\_from}{$i$}
      \ForAll{$j$ in $\cN(i)$}
      \State $j_n=\mathrm{Mod}(j,n)$
	\If{$j_n\notin\bnv$}
	  \If{$a_{j_n}=1$}
	    \State append $(i,j)$ to $\tree$
	    \State append $j_n$ to $\bnv$
	    \State  execute {\sc dfs\_from}$(j_n)$
	  \Else
	     \State append $(i,j)$ to $\damp$
	     \State $a_{j_n} = a_{j_n}-1$
	  \EndIf
       \EndIf	
     \EndFor
   \EndFunction
  \end{algorithmic}
}
\end{minipage}\hfill\vrule\hfill
\begin{minipage}[t]{.5\textwidth}
{\footnotesize
\algrenewcommand\algorithmicindent{1.15em}%
\begin{algorithmic}[1]
  \Statex{\hspace{-0.5cm}\sc Tree to Parking Function Algorithm (ad.)}
  \Statex{\bf Input:}\  Spanning tree $T$ rooted\\at $r$ with edges directed away
  from root.
  \State $\bnv =\{r\}$
  \State $\damp = \{\,\}$
  \State $\ba=(1,\dotsc,1)$
  \State execute {\sc tree\_from$(r)$}
  \Statex{\bf Output:}\  {$\ba\colon V\setminus\{r\}\to\N$}
  \Statex{\hspace{-0.5cm}\sc auxiliary function}
  \Function{\sc tree\_from}{$i$}
      \ForAll{$j$ in $\cN(i)$}
      \State $j_n=\mathrm{Mod}(j,n)$
      \If{$j_n\notin\bnv$}
	\If{$(i,j)$ is an edge of $T$}
	    \State append $j_n$ to $\bnv$
	    \State execute {\sc tree\_from}$(j_n)$
	 \Else
	   \State {$a_{j_n} = a_{j_n}+1$}
	   \State append $(i,j)$ to $\damp$
	 \EndIf
      \EndIf
    \EndFor
  \EndFunction
  \Statex
  \end{algorithmic}
}\end{minipage}
\caption{$\dfs$-Burning Algorithm and inverse}
\label{algs}
\end{figure}

\begin{example}
We apply the $\dfs$-Burning Algorithm  to $\ba=4213\in[4]^4$ with the three different graphs associated with $n=4$.
Actually, $\ba$ is a parking function ---that is, a label of a region of $\cA^2_4=\Shi_4$--- since $\bat{2}=4321$ and $1\in Z(\bat{2})=[4]$,
but  neither a label of a region of $\cA_4^3$ nor of $\cA_4^4=\Ish_4$, because $\bat{3}=4231$,
$\bat{4}=\ba=4213$,
$1\notin Z(\bat{3})=\{2,4\}$, and $1\notin Z(\bat{4})=\{2,3\}$.

In the first case, where
$\cN=${\small$\big\langle \langle 4,3,2,1\rangle,\langle4,3,2\rangle,\langle4,3,1\rangle,
\langle4,2,1\rangle,\langle3,2,1\rangle\big\rangle$} (cf. the left table in the bottom of Figure~\ref{examplex}),
when the algorithm is applied with  $\cG=\overline{\cG_4^2}$ to $\ba$,  it calls $\text{\sc dfs\_from}(i)$ with $i=0$, assigns $j=4$ and then, since
$a_j\neq1$, $(0,4)$ is joined to $\damp$. This is represented on the left-hand table below with the inclusion of $0_1$ in the top box of column $4$. 
Next assignment, $j=3$. Since now  $a_3=1$, $(0,3)$ is joined to $\tree$ and {\sc dfs\_from} is called with $i=3$. Then, $0_2$ is written in the only box of column $3$.  
At the end, $\bnv=${\small$\langle0,3,2,4,1\rangle$}, which proves that $4213$ is a $\cG^2_4$-parking function, that is, a standard parking function in dimension $4$. The respective spanning tree may be defined by the collection of arcs,
$\tree=${\small$\langle(0,3),(0,2),(2,4),(0,1)\rangle$}.

\begin{figure}[ht]
{\footnotesize
$$\begin{array}{c|c|c}
\strut\hspace{.35cm}
\begin{array}{|l||l||l|l|l|l|}
\hline
i&0&1&2&3&4\\
\hline
&\vbox to.375cm{\vfill}4&4&4&4&-\\
&\vbox to.375cm{\vfill}3&3&3&-&3\\
&\vbox to.375cm{\vfill}2&2&-&2&2\\
\rotatebox{90}{\rlap{\tiny$\cN(i)$}}&\vbox to.375cm{\vfill}1&-&1&1&1\\
\hline
\end{array}\hspace{.35cm}
&\strut\hspace{.35cm}
\begin{array}{|l||l||l|l|l|l|}
\hline
i&0&1&2&3&4\\
\hline
&\vbox to.375cm{\vfill}4&8\,4&-&4&-\\
&\vbox to.375cm{\vfill}3&7\,3&-&-&3\\
&\vbox to.375cm{\vfill}2&2&-&2&2\\
\rotatebox{90}{\rlap{\tiny$\cN(i)$}}&\vbox to.375cm{\vfill}1&-&1&1&1\\
\hline
\end{array}\hspace{.35cm}
&\strut\hspace{.35cm}
\begin{array}{|l||l||l|l|l|l|}
\hline
i&0&1&2&3&4\\
\hline
&\vbox to.375cm{\vfill}4&12\ 8\,4&-&-&-\\
&\vbox to.375cm{\vfill}3&7\,3&-&-&3\\
&\vbox to.375cm{\vfill}2&2&-&2&2\\
\rotatebox{90}{\rlap{\tiny$\cN(i)$}}&\vbox to.375cm{\vfill}1&-&1&1&1\\
\hline
\end{array}\hspace{.35cm}\strut\\[40pt]
\begin{array}{cccc}
\cline{1-1}
\cx{3_5}&&&\\
\cline{1-1}\cline{4-4}
\cx{4_8}&&&\cx{0_1}\\
\cline{1-2}\cline{4-4}
\cx{2_9}&\cx{3_4}&&\cx{3_3}\\
\hline
\hline
\cx{\!0_{10}\!}&\cx{0_6}&\cx{0_2}&\cx{2_7}\\
\hline
{\scriptstyle1}&{\scriptstyle2}&{\scriptstyle3}&{\scriptstyle4}
\end{array} 
&\begin{array}{cccc}
\cline{1-1}
\cx{3_5}&&&\\
\cline{1-1}\cline{4-4}
\cx{2_7}&&&\cx{0_1}\\
\cline{1-2}\cline{4-4}
\cx{0_8}&\cx{3_4}&&\cx{3_3}\\
\hline
\hline
\cx{}&\cx{0_6}&\cx{0_2}&\cx{}\\
\hline
{\scriptstyle1}&{\scriptstyle2}&{\scriptstyle3}&{\scriptstyle4}
\end{array} 
&\begin{array}{cccc}
\cline{1-1}
\cx{3_4}&&&\\
\cline{1-1}\cline{4-4}
\cx{2_6}&&&\cx{0_1}\\
\cline{1-2}\cline{4-4}
\cx{0_7}&\cx{3_3}&&\cx{}\\
\hline
\hline
\cx{}&\cx{0_5}&\cx{0_2}&\cx{}\\
\hline
{\scriptstyle1}&{\scriptstyle2}&{\scriptstyle3}&{\scriptstyle4}
\end{array}\\[30pt]
\overline{\cG_4^2}=\overline{K_4}&\overline{\cG^3_4}&\overline{\cG_4^4}
\end{array}$$}
\caption{Lists of neighbours and execution of the $\dfs$-Burning Algorithm}
\label{examplex}
\end{figure}

We believe that now the content of the tables is self-explanatory. Just note that the entry $i_k$ in column $j$ means that arc $(i,j)$ is the $k$.th arc to be inserted \footnote{Perhaps with label $(i,j+m\,n)$.}.
Note also that the elements $i\in[n]$ of the bottom row are those for which $a_i=1$, and thus represent elements from $\tree$, whereas the remaining entries represent elements from $\damp$.

Finally, note that the algorithm  runs in the second graph by choosing the same arcs up to the seventh arc, which is not $(2,4)$ since $4\notin\cN(2)$ in this graph. Since $\bnv\neq\{0,1,\dotsc,4\}$ at the end of the execution for the two last graphs, we verify that $4213$ is neither a label of the regions of $\cA^3_4$ nor an Ish-parking function (in fact, $1\notin Z(4213)=\{2,3\}$).
\end{example}

\begin{lemma}\label{centro}
Let $\ba\in[n]^n$ be the input of the $\dfs$-Burning Algorithm
applied to $\bGk$ ($2\leq k\leq n$) as defined above, and suppose that, at the end of the execution, the list of burnt vertices is
 $\bnv=\langle 0\!\!=\!\!i_0,i_1,\dotsc,i_m\rangle$.
 Suppose $i_p:=\min\{i_1,\dotsc,i_m\}<k$. 
Then either $i_p=1$ or $p=m$. In any case, if $\bak=\ba\circ\pi$ for $\pi\in\mathfrak{S}_m$ defined as in the beginning of Section~\ref{ppf}
$$Z(\bak)=\{\pi(i_1),\dotsc,\pi(i_p)\}\,.$$
\end{lemma}
\begin{proof}
Note that:
\begin{itemize}
\item The value of $a_{i_j}$ is one when $i_j$ is appended to $\bnv$, at Line~11; it has decreased one unit in previous calls of   {\sc dfs\_from$(i)$}, exactly when $i=i_\ell$  and $i_j\in\cN(i_\ell)$ for some $\ell<j$. Hence, 
$$\forall\,j\in[m]\,,\quad a_{i_j}\leq j\,.$$
\item If $1<i_j<k$ then $i_{j+1}<i_j<k$, since $i_{j+1}\in \cN(i_j)$ and $i_j<k$.
\item If $i_{j+1},i_j\geq k$, and $i_{j+1}>i_j$, then $a_{i_{j+1}}\leq j$ and $a_{i_j}\leq j+1$ since $a_{i_j}\leq j$. 
\end{itemize}
Hence, if $\ell_j=\pi(i_j)$ for every $j\in[m]$, then
\begin{align*}
&\ell_m\leq\ell_{m-1}\leq\dotsb\leq\ell_1\,;\\
&\forall\,j\in[m]\,,\quad a_{\ell_j}\leq j\,.
\end{align*}

For the converse, 
note that, by definition of $\bGk$, if $j\in\cN(p)$ for some $p>1$, $m\neq p$, and $m<j$, then also $m\in\cN(p)$.
Thus, if at the end of the execution $j\in\bnv$, $m<j$ and $a_m\leq a_j+1$, then also $m\in\bnv$.
\end{proof}

\section{Main Theorem}

\begin{theorem}\label{main}
The $\cG^k_n$-parking functions are exactly the $k$-partial parking functions. Their number is
$$(n+1)^{n-1}\,.$$
\end{theorem}
We know that there are $(n+1)^{n-1}$ regions in the $\cA^k_n$ arrangement of hyperplanes, which are bijectively labelled by the  
$\cG^k_n$-parking functions \cite[Theorem 3.7]{DGO3}.

Hence, all we have to prove is the first sentence. This is an immediate consequence of the following  Lemma~\ref{Ishlabels} and of the fact that the $\cG$-parking functions are those functions for which the $\dfs$-Burning Algorithm burns all vertices during the whole execution.

\begin{lemma}\label{Ishlabels}
Let $\ba\in[n]^n$ be the input of the $\dfs$-Burning Algorithm
applied to $\bGk$ ($2\leq k\leq n$) as defined above, and consider $\bnv=\langle 0\!\!=\!\!i_0,i_1,\dotsc,i_m\rangle$ 
at the end of the execution.
Then the following statements are equivalent:
\renewcommand{\theenumi}{\ref{Ishlabels}.\arabic{enumi}}
\renewcommand{\labelenumi}{\theenumi. }
\begin{enumerate}
\item\label{ia} $\ba$ parks every element of $[k,n]$ and $i_p=1$ for some $1\leq p\leq m$;
\item\label{ib}  $\ba$ is a $k$-partial parking function; 
\item\label{ic} as a set, $\bnv=\{0\}\cup[n]$ or, equivalently, $m=n$.
\end{enumerate}
\end{lemma}
\renewcommand{\theenumi}{\arabic{enumi}}
\renewcommand{\labelenumi}{\theenumi. }

\begin{proof}$\ $\\[-18pt]

\medskip\noindent
\eqref{ia}$\implies$\eqref{ib}. \enspace
Since $\ba$ parks all the elements of $[k,n]$, it is sufficient to show that $1\in Z(\bak)$, which follows from Lemma~\ref{centro}.

\medskip\noindent
\eqref{ib}$\implies$\eqref{ic}. \enspace
Suppose that $1$ belongs to the  centre of $\bak$ but there is a greatest element $j\in[n]$ which is not in $\bnv$ at the end of the execution. Suppose first that $j<k$. Then, during the execution of the algorithm (more precisely, during the execution of Line~14)
the value of $a_j$ has decreased once for $i=0$ (that is, as a neighbour of $0$), once for each value of $i>j$ (in a total of $n-j$), since $i\in\bnv$ by definition of $j$, and $j-1$ times for $i=1$, and is still greater than zero. Hence $a_j>n$, 
which is absurd.

Now, suppose that $j\geq k$, and let 
\begin{align*}
&\alpha=\min\big\{a_i\mid i\notin\bnv\cap[k,n]\big\}\,\\
&p=\min\big\{q\in[k,n]\mid a_q=\alpha\big\}\quad\text{and}\\
&A=\big\{q\in[k,n]\mid a_q<\alpha\big\}\,,
\shortintertext{so that}
&\begin{cases}|A|\geq \alpha-k\quad\text{(since $\ba$ parks all the elements of $[k,n]$);}\\A\subseteq\bnv.\end{cases}
\end{align*}
Again, during the execution of Line~14
the value of $a_p=\alpha$ has decreased once for $i=0$, once for each value of $i\neq p$ in $\bnv\cap[k,n]\supseteq A$, and $k-1$ times for $i=1$, and is still greater than zero. This means that $\alpha-1-(\alpha-k)-(k-1)>0$, which is not possible.

\medskip\noindent
\eqref{ic}$\implies$\eqref{ia}. \enspace
Contrary to our hypothesis, we admit that all the elements of $[n]$ belong to $\bnv$
at the end of the execution, but that for some $j\in[k,n]$
\begin{align*}&A_j=\big\{q\in[k,n]\mid a_q>j\big\}
\shortintertext{verifies}
&|A_j|\geq n-j+1\,.
\end{align*}
Remember that $\bnv=\langle 0\!\!=\!\!i_0,i_1,\dotsc,i_n\rangle$ is the ordered list of burnt vertices at the end of the execution and let
$$r=\min\big\{q\in[k,n]\mid i_q\in A_j\big\} \text{ and } p=i_r\,.$$
Then  $\alpha=a_p>j$. Since $p\in\bnv$, the number of elements of form $(i,p)$ of the set $\damp\cup\tree$  (which is equal to $\alpha$)  must be greater than $j$. But when $p$ was burned, at most $(n-k+1)-(n-j+1)=j-k$ elements of $[k,n]$ different from $p$ were already burned, and even if
$0$ and $1$ were also burned, the number of edges could not be greater than $(j-k)+1+(k-1)=j$, a  contradiction.\qedhere
\end{proof}

\noindent\emph{Acknowledgements.}\enspace
This work was partially supported by CMUP (UID/MAT/00144/ 2013)
  and CIDMA (UID/MAT/04106/2013), which are funded by FCT (Portugal)
  with national (ME) and European structural funds through the
  programs FEDER, under the partnership agreement PT2020.

\end{document}